\documentclass{amsart}
\usepackage{amsmath, amsthm, amssymb}
\usepackage{graphicx}

\usepackage[utf8]{inputenc}
\usepackage[author-year]{amsrefs}

\newtheorem{theorem}{Theorem}
\newtheorem{subtheorem}{Theorem}

\newtheorem{lemma}[theorem]{Lemma}

\theoremstyle{definition}
\newtheorem{definition}[theorem]{Definition}
\theoremstyle{remark}
\newtheorem{remark}[theorem]{Remark}

\newcommand{\conv}[1]{\mathrm{Conv}(#1)}
\newcommand{\CONV}[1]{\mathrm{CONV}_{#1}}
\DeclareMathOperator*{\Avg}{Avg}
\newcommand{\avg}[1]{\Avg_{#1}}
\newcommand{\vol}{\mathrm{Vol}}
\newcommand{\dd}{\,\mathrm{d}}

\newcommand{\defeq}{\stackrel{\scriptsize\mathrm{def}}{=}}

\author{Gregorio Malajovich}
\thanks{This paper was written while visiting the Simons Institute for the Theory of Computing in the
University of California at Berkeley. This visit was funded by CAPES (Coordenação de 
Aperfeiçoamento de Pessoal de Nível Superior, Brazil. Proc. BEX 2388/14-6).}
\address{Departamento de Matemática Aplicada, Instituto de Matemática, Universidade Federal do
Rio de Janeiro. Caixa Postal 68530, Rio de Janeiro RJ 21941-909, Brasil.}\email{gregorio.malajovich@gmail.com}
\title{Average mixed volume under projection}
\date{\today}
\begin{document}
\maketitle
\begin{abstract}
The average mixed volume of a random projection of $d$ convex bodies in $\mathbb R^n$ is
bounded above in terms of a quermassintegral.
\end{abstract}

\section{Introduction}

The {\em mixed volume} of the $n$-tuple of convex bodies $(\mathcal A_1, \dots, \mathcal A_n)$,
$\mathcal A_i \subset \mathbb R^n$ is defined as
\[
V(\mathcal A_1, \dots, \mathcal A_n)
\defeq
\frac{1}{n!}
\frac{\partial^n}
{\partial t_1 \partial t_2 \cdots \partial t_n}
\mathrm{Vol}(t_1 \mathcal A_1 + \cdots + t_n \mathcal A_n) 
\]
where $t_1, \dots, t_n \ge 0$ and the derivative is taken at $t=0$. 
The normalization factor $1/n!$ ensures that $V(\mathcal A, \dots, \mathcal A) = \vol(\mathcal A)$.
In that sense the mixed volume generalizes the usual volume.
\par
The definition above 
is still valid if we allow the $\mathcal A_i$ to have empty volume, so in this paper
we will only require the $\mathcal A_i$ to be closed convex sets.
\medskip
\par
The mixed volume was introduced by \ocite{Minkowski} for $n=3$. If
$\mathcal A \subset \mathbb R^3$ is a convex body, the Steiner
formula
\[
\vol(\mathcal A + \epsilon B^3) =
\vol(\mathcal A) 
+
S \epsilon 
+
\pi B \epsilon^2
+
\frac{4}{3} \pi \epsilon^3
\]
implies that 
\[
V(\mathcal A, \mathcal A, B^3) = 3S 
\hspace{2em}\text{and}\hspace{2em}
V(\mathcal A, B^3, B^3)= 3 \pi B
\] 
with $S$ is the total area of $\partial A$ and $B$ its total mean curvature
(assuming $\partial A$ is smooth).
The quantities $V(\mathcal A, \dots, \mathcal A, B^n, \dots, B^n)$ are known as {\em quermassintegrals}. There is a vast
literature on recovering quantities such as the area of $\partial \mathcal A$ from the area of projections
of $\mathcal A$.
\medskip
\par
However, most of the literature deals with mixed volumes for two different
convex bodies, one being the ball $B^n$. Mixed volumes with many possibly
different convex bodies seem to arise mostly in connection with counting
zeros of sparse polynomial systems or obtaining a starting system for
path-following. The main result in this note arised during the complexity 
analysis of such an algorithm.
\medskip
\par
Let $G(d,n)$ denote the Grassmannian manifold of $d$-dimensional subspaces of $\mathbb R^n$. It is
endowed with the unique probability measure invariant under the orthogonal group $O(n)$.
In this note we prove the following result:
\begin{theorem}\label{main} Let $\mathcal A_1, \dots, \mathcal A_d$ be compact convex sets in $\mathbb R^n$. 
Then,
\[
\avg{P \in G(d,n)} V(P(\mathcal A_1), \cdots, P(\mathcal A_d))  \le
\frac
{\vol(B^d)}
{\vol(B^n)}
V(\mathcal A_1, \cdots, \mathcal A_d, B^n, \cdots, B^n) 
.
\]
\end{theorem}

By replacing $\mathcal A_i$ with $B^n$ and $P(\mathcal A_i)$ by $B^d$, one quickly checks that the
bound above is sharp.

\section{Proof of the theorem}

Before proving theorem \ref{main}, we will restate it in a more convenient form.
If $Q \in O(n)$, let $q_j$ denote the $j$-th row of $Q$. If $Q$ is uniformly distributed
in $O(n)$ with respect to the Haar measure, the projection $P: \mathbb R^n \rightarrow \mathbb R^d$,
\[
P: x \mapsto 
\left[
\begin{matrix} 
q_1 x \\ \vdots \\ q_d x
\end{matrix}
\right]
\]
is uniformly distributed in $G(d,n)$. We have
\[
d!\, V( P(\mathcal A_1), \dots, P(\mathcal A_d)) = n!\, V(\mathcal A_1, \dots, \mathcal A_d, [0, q_{d+1}], \cdots, [0, q_n])
.
\]

Therefore, what we need to prove is the following restatement of Theorem~\ref{main}:

\begin{subtheorem} Let $\mathcal A_1, \dots, \mathcal A_d$ be compact convex sets in $\mathbb R^n$. 
Then,
\[
\begin{split}
\avg{Q \in O(n)} V(\mathcal A_1, \dots, \mathcal A_d, [0,q_{d+1}],& \dots, [0,q_{d+n}])  \le\\
&
\le 
\frac
{d!\, \vol(B^d)}
{n!\, \vol(B^n)}
\, V(\mathcal A_1, \cdots, \mathcal A_d, B^n, \cdots, B^n) 
.
\end{split}
\]
\end{subtheorem}

\begin{remark} Explicitly,
\[
\frac
{d!\, \vol(B^d)}
{n!\, \vol(B^n)}
=
\frac{1}{(2\sqrt{\pi})^{n-d}} \frac {\Gamma\left(\frac{d+1}{2}\right)} {\Gamma\left(\frac{n+1}{2}\right)} 
.
\]
This is an immediate consequence of the duplication formula~\cite{Abramowitz-Stegun}*{formula 6.1.18},
\[
\Gamma(2z) = \frac{1}{\sqrt{2 \pi}} 2^{2z-1/2}\, \Gamma(z)\, \Gamma(z+1/2).
\]
\end{remark}

Let $\CONV{n}$ denote the set of all compact convex subsets of $\mathbb R^n$.
This set is closed under Minkowski sum and under multiplication by non-negative
real numbers. In that sense, $\CONV{n}$ is an `algebra' over the non-negative
real numbers. We recall the following definitions:

\begin{definition}
A function $F: \CONV{n} \rightarrow \mathbb R^{+}$ is linear if for all 
$t_1, t_2 \in \mathbb R^+$ and for all $\mathcal A_1, \mathcal A_2 \in \CONV{n}$,
$F(t_1 \mathcal A_1 + t_2 \mathcal A_2) = t_1 F(\mathcal A_1) + t_2 F(\mathcal A_2)$.
\end{definition}

\begin{definition}
A function $F: \CONV{n} \rightarrow \mathbb R^{+}$ is monotone if 
$\mathcal A \subseteq \mathcal B$ implies that $F(\mathcal A) \le F(\mathcal B)$.
\end{definition}

The mixed volume $V(\mathcal A_1, \dots, \mathcal A_n)$ is linear and
monotone in each of the variables. The key for proving theorem~\ref{main} is the following Lemma,
that I could not find it in the literature.

\begin{lemma}\label{lem:needle} Let  $F: \CONV{k} \rightarrow \mathbb R^{+}$ be linear and monotone.
Then
\[
\avg{c \in S^{k-1}} ( F([0,c])) \le r_k F(B^k)
\]
with $r_k = \frac{\Gamma(k/2)}{\sqrt{\pi} (k-1)\,  \Gamma\left( \frac{k-1}{2} \right)}$.
This constant $r_k$ is the smallest with such a property.
\end{lemma}

\begin{proof}
For every $\epsilon > 0$, there is a closed covering of $S^{k-1}$ into measurable subsets of
radius $\epsilon$, say $S^{k-1} = \cup_{\lambda \in \Lambda} V_{\lambda}$, and
such that $V_{\lambda} \cap V_{\lambda'}$ has measure zero for $\lambda \ne \lambda'$.
The sets $W_{\lambda} = \conv{\{0\} \cup V_{\lambda}}$ are a closed covering of
$B^k$ with the same property. 

\begin{align*}
\avg{c \in S^{k-1}} ( F([0,c])) &= \frac{1}{\vol{S^{k-1}}} \int_{S^{k-1}} F([0,c]) \ \dd S^{k-1}(c) && \text{by definition,}\\
&=\frac{1}{\vol{S^{k-1}}} \sum_{\lambda \in \Lambda} \int_{V_{\lambda}} F([0,c]) \ \dd S^{k-1}(c) && \text{trivially,}\\
&\le\frac{1}{\vol{S^{k-1}}} \sum_{\lambda \in \Lambda} \int_{V_{\lambda}} F(W_{\lambda}) \ \dd S^{k-1}(c) && \text{by monotonicity,}\\
&\le\frac{1}{\vol{S^{k-1}}} \sum_{\lambda \in \Lambda} (\vol{V_{\lambda}}) F(W_{\lambda}) && \text{trivially,}\\
&\le\frac{1}{\vol{S^{k-1}}} F\left(\sum_{\lambda \in \Lambda} (\vol{V_{\lambda}}) W_{\lambda}\right) && \text{by linearity}\\
\end{align*}
After an orthogonal change of coordinates, assume that $\rho(\epsilon) \mathrm e_1$ is a point of largest norm
in $\sum_{\lambda \in \Lambda} (\vol{V_{\lambda}}) W_{\lambda}$. In particular, there are $x_{\lambda} \in W_{\lambda}$
such that $\rho(\epsilon) \mathrm e_1 = \sum_{\lambda \in \Lambda} (\vol{V_{\lambda}}) x_{\lambda}$.
So,
\[
\rho(\epsilon) = \sum_{\lambda \in \Lambda} (\vol{V_{\lambda}}) (x_\lambda)_1 \le  
\sum_{ (x_{\lambda})_1 \ge 0} \vol{V_{\lambda}} \frac{(x_\lambda)_1}{\|x_\lambda\|}.
\]

Passing to the limit,
\begin{align*}
\lim_{\epsilon \rightarrow 0}\rho(\epsilon) &= 
\int_{S^{k-1} \cap [y_1 \ge 0]} y_1 \dd S^{k-1}(y) \\
               &= 
\int_{S^{k-1} \cap [y_1 \ge 0]} \mathrm e_1 \cdot  \dd n(y) && \text{interpreting as a flow,}\\
               &= 
\int_{B^{k} \cap [y_1 \ge 0]} \mathrm{div}(\mathrm e_1)  \dd B^{k}(y) 
-
\int_{B^{k-1}} \mathrm e_1 \cdot  \dd n(y)&&\text{by Gauss theorem,} \\
               &=
0 - \int_{B^{k-2}} -1 \dd B^{k-1}(y) \\
&= \vol{B^{k-1}}
\end{align*}

By monotonicity of $F$ and then linearity, we conclude that
\[
\avg{c \in S^{k-1}} ( F([0,c])) \le 
\frac{1}{\vol S^{k-1}} F( \lim_{\epsilon \rightarrow 0} \rho(\epsilon) \ B^k)
\le
r_k F(B^k) 
\]
with $r_k = \vol{B^{k-1}} / \vol{S^{k-1}}$. Using
the well-known formulas
\[
\vol B^{k} = \frac{2 \pi^{k/2}}{k\  \Gamma\left( \frac{k}{2} \right)}
\hspace{3em}
\text{and}
\hspace{3em}
\vol S^{k-1} = \frac{2 \pi^{k/2}}{\Gamma\left( \frac{k}{2} \right)}
\]
we deduce that
\[
r_k = \frac{\Gamma\left( \frac{k}{2} \right)}{\sqrt{\pi} (k-1)\  \Gamma\left( \frac{k-1}{2} \right)}
.
\]
\medskip
\par
To establish sharpness, we consider the operator $F(X) = V(B^k, \dots, B^k, X)$. Because
of the properties of the mixed volume,
\[
F( [0, c] ) = V(B^{k} \cap c^{\perp}, \dots, B^{k} \cap c^{\perp},[0,c])
.
\]
Thus,
\[
\avg{c \in S^{k-1}} F([0,c]) = F( [0, \mathrm e_k] ) = V(B^{k-1}\times\{0\}, \dots, B^{k-1} \times\{0\},[0,\mathrm e_k])
.
\]

This last value is by definition $\frac{1}{k!}$ times the coefficient in $t_1 t_2 \cdots t_k$ of
\[
\vol{ (t_1+\cdots+t_{k-1}) B^{k-1} + t_k[0,\mathrm e_k] } = (t_1+\cdots+t_{k-1})^{k-1}t_k \vol B^{k-1}.
\]
Thus,
\[
\avg{c \in S^{k-1}} F([0,c]) = \frac{1}{k} \vol B^{k-1} 
\]
while the predicted value is
$r_k \vol B^k = \frac{\vol{B^{k-1}}}{\vol{S^{k-1}}} \vol{B^k} = \frac{1}{k}\vol B^{k-1}$.
\end{proof}

\begin{proof}[Proof of Theorem \ref{main}]
The last row $q_n$ of $Q$ is uniformly distributed in $S^{n-1}$,
so
\[
\begin{split}
\avg{Q \in SO(n)} 
\left( V(\mathcal A_1,\right.\left. \cdots, \mathcal A_d,\right.&\left. [0,q_{d+1}], \cdots, [0,q_n]) 
\right)=
\\
&=\avg{q_n \in S^{n-1}} \left( 
\avg{}
\left(
V(\mathcal A_1, \cdots, \mathcal A_d, [0,q_{d+1}], \cdots, [0,q_n]) 
\right) \right)
\end{split}
\]
where the inner average is taken over orthonormal frames $(q_{d+1}, \dots, q_{n-1})$
that are orthogonal to $q_n$.
By Lemma~\ref{lem:needle},
\[
\begin{split}
\avg{Q \in SO(n)} 
\left( V(\mathcal A_1,\right.&\left. \cdots, \mathcal A_d, [0,q_{d+1}], \cdots, [0,q_n]) 
\right)=
\\
&=
r_{n} 
\avg{}
\left(
V(\mathcal A_1, \cdots, \mathcal A_d, [0,q_{d+1}], \cdots, [0,q_{n-1}], B^n) 
\right).
\end{split}
\]

Repeating the same argument, 

\[
\begin{split}
\avg{Q \in SO(n)} V(\mathcal A_1, \cdots, \mathcal A_d, &  [0,q_{d+1}], \cdots, [0,q_n]) 
\\
&\le
r_{d+1} r_{d+2} \cdots r_n 
V(\mathcal A_1, \cdots, \mathcal A_d, B^n, \cdots, B^n) 
.
\end{split}
\]

Finally,
\[
r_{d+1} r_{d+2} \cdots r_n 
=
\frac{\Gamma(n/2)\, \Gamma(d)}{\pi^{\frac{n-d}{2}}\, \Gamma(d/2)\, \Gamma(n)}
\]
\end{proof}

\end{document}